\documentclass[11pt]{amsart}

\usepackage{amsmath,amssymb,latexsym,soul,cite,mathrsfs,amsfonts}

\usepackage{color,enumitem,graphicx}
\usepackage[colorlinks=true,urlcolor=blue,
citecolor=red,linkcolor=blue,linktocpage,pdfpagelabels,
bookmarksnumbered,bookmarksopen]{hyperref}
\usepackage[english]{babel}
\usepackage{comment}

\usepackage[left=2.9cm,right=2.9cm,top=2.8cm,bottom=2.8cm]{geometry}
\usepackage[hyperpageref]{backref}

\usepackage[colorinlistoftodos]{todonotes}
\makeatletter
\providecommand\@dotsep{5}
\def\listtodoname{List of Todos}
\def\listoftodos{\@starttoc{tdo}\listtodoname}
\makeatother

\usepackage{verbatim} 

\numberwithin{equation}{section}

\newtheorem{theorem}{Theorem}[section]
\newtheorem{proposition}[theorem]{Proposition}
\newtheorem{lemma}[theorem]{Lemma}
\newtheorem{corollary}[theorem]{Corollary}
\newtheorem{remark}{Remark}
\newcommand\restr[2]{{
  \left.\kern-\nulldelimiterspace 
  #1 
  \vphantom{\big|} 
  \right|_{#2} 
  }}

\title[]{On the extreme value of the Nehari manifold method for a class of Schr\"{o}dinger equations with indefinite weight functions }

\author[J. C. de Albuquerque]{Jos\'e Carlos de Albuquerque$^{1}$}

\author[K. Silva]{Kaye Silva$^{2}$}

\address[J. C. de Albuquerque]{\newline\indent
	Departamento de Matem\'atica.   
	\newline\indent 
Universidade Federal de Pernambuco,   
	\newline\indent 
	50670-901 Recife-PE, Brazil}
\email{\href{mailto:joserre@gmail.com}{joserre@gmail.com}, \href{mailto:jc@dmat.ufpe.br}{jc@dmat.ufpe.br}}

\address[K. Silva]{\newline\indent
	Instituto de Matem\'atica e Estat\'istica.   
	\newline\indent 
	Universidade Federal de Goi\'as,
	\newline\indent
	74001-970, Goi\^ania, GO, Brazil}
\email{\href{mailto:kayeoliveira@hotmail.com}{kayeoliveira@hotmail.com}, \href{mailto:kaye_0liveira@ufg.br}{kaye\_0liveira@ufg.br}}

\thanks{The first author was partially supported by Capes/Brazil. The second author was partially supported by CNPq/Brazil under Grant [408604/2018-2]. 
}


\begin{document}
	
\maketitle

\vspace{-0,5cm}

\hspace{2cm} $^{1}${\small Departamento de Matem\'atica, Universidade Federal de Pernambuco,}\\
\vspace{-0,5cm}

\hspace{2,3cm}{\small Recife-PE, Brazil}

\hspace{2cm} $^{2}${\small Instituto de Matem\'atica e Estat\'istica, Universidade Federal de Goi\'as,}\\
\vspace{-0,5cm}

\hspace{2,1cm} {\small Goi\^ania-GO, Brazil}

\begin{abstract}
	In this work we are concerned with the following class of equations
	 \[
	  -\Delta_p u -\lambda h(x)|u|^{p-2}u=f(x)|u|^{\gamma-2}u, \quad \mbox{in } \mathbb{R}^N,
	 \]
	involving indefinite weight functions. The existence of solution may depend on the parameter $\lambda$. We analyze the extreme value $\lambda^{*}$ and study its relation with the Nehari manifold. Our goal is to establish the existence of two solutions when $\lambda>\lambda^{*}$. This work extends and complements the results obtained by J. Chabrowski and D.G. Costa [Comm. Partial Differential Equations 33 (2008), 1368--1394] 
\end{abstract}

\section{Introduction}

In this work we study the following class of problems

\begin{equation}\label{p}
\left\{
\begin{aligned}
-\Delta_p u -\lambda &h(x)|u|^{p-2}u=f(x)|u|^{\gamma-2}u  &&\mbox{in}\ \ \mathbb{R}^N, \\
&u\in D^{1,p}(\mathbb{R}^N)\cap L^\gamma(\mathbb{R}^N),  &&
\end{aligned}
\right.
\end{equation}
where $p\in (1,\infty)$, $\gamma\in (p,p^*)$, $\lambda$ is a real parameter, $h\in L^{\frac{N}{p}}(\mathbb{R}^N)\cap L^\infty(\mathbb{R}^N)$, $f\in L^\infty(\mathbb{R}^N)$ and $\Delta_p$ is the $p$-Laplacian operator. Moreover, $D^{1,p}(\mathbb{R}^N)$ is the closure of $C_0^\infty(\mathbb{R}^N)$ with respect to the norm $\|u\|_{D^{1,p}(\mathbb{R}^{N})}^{p}=\int|\nabla u|^p$. We denote $E\equiv D^{1,p}(\mathbb{R}^N)\cap L^\gamma(\mathbb{R}^N)$ and equip $E$ with the norm
\begin{equation*}
\|u\|= \left[\int |\nabla u|^p+\left(\int |u|^{\gamma}\right)^{\frac{p}{\gamma}}\right]^{\frac{1}{p}}.
\end{equation*}
Consider the eigenvalue problem 

\begin{equation}\label{pe}
\left\{
\begin{aligned}
-\Delta_p u &= \lambda h|u|^{p-2}u
&\mbox{in}\ \ \Omega \\
&u\in D^{1,p}(\Omega) 
\end{aligned}
\right.,
\end{equation}
where $\Omega\subset \mathbb{R}^N$ is an open set. We denote the first eigenvalue of \eqref{pe}, when it exists, by $\lambda_1$. 


There is a large literature concerning existence results for several classes of problems related to \eqref{p} and we refer to the readers, for example, \cite{ouyang,alamadel,tarantello,bere,bere2,costa,costa2,yavkay,tarantello2,cingo,giaco} and references therein. In the work of Ouyang \cite{ouyang} the author has studied the class of problems 
\begin{equation}\label{pat}
\left\{
\begin{aligned}
-\Delta_p u -\lambda &h(x)|u|^{p-2}u=f(x)g(u)  &&\mbox{in}\ \ \Omega, \\
&u\in W_0^{1,p}(\Omega),  &&
\end{aligned}
\right.
\end{equation}	
in the particular case where $p=2$, $\Omega\subset\mathbb{R}^{N}$ is a bounded domain, $h(x)=1$ and $g(u)=|u|^{\gamma-2}u$. Ouyang proved the existence of $\lambda_b>0$ such that problem \eqref{pat} has at least two positive solutions whenever $\lambda\in(\lambda_1,\lambda_b)$, at least one positive solution for $\lambda=\lambda_b$ and does not admit positive solutions for $\lambda>\lambda_b$. Later on in Alama-Tarantello \cite{tarantello}, that result was generalized by considering more general hypothesis on $g$ and $h(x)=1$. Precisely, it was introduced the notion of ``thickness". When $\Omega=\mathbb{R}^N$ and $p=2$, problem \eqref{pat} was studied in Costa-Tehrani \cite{costa2} and they proved the existence of two solutions whenever $\lambda$ is close to $\lambda_1$. 

In Il'yasov-Silva \cite{yavkay}, problem \eqref{pat} was studied when $\Omega$ bounded and $g(u)=|u|^{\gamma-2}u$. By following the ideas introduced in \cite{yav2}, the authors were able to provide existence of solutions only by variational methods, by introducing the so-called extreme parameter $\lambda^*$ and $\varepsilon>0$ for which problem \eqref{pat} has at least two positive solutions for $\lambda\in(\lambda_1,\lambda^*+\varepsilon)$. When $\Omega=\mathbb{R}^N$ and $g(u)=|u|^{\gamma-2}u$ problem \eqref{pat} was studied in \cite{costa} where the authors proved the existence of two positive solutions for $\lambda$ close to $\lambda_1$. 

 Motivated by \cite{yav2} and \cite{yavkay}, our main goal is to extend and complement the results of \cite{costa}, by showing existence of two positive solutions for $\lambda\in(\lambda_1,\lambda^*+\varepsilon)$. As will become clear in the work, there is a substantial difference when one tries to find solutions in $(\lambda_1,\lambda^*)$ or $\lambda\ge \lambda^*$. In fact, the main technique employed in \cite{costa} to find solutions when $\lambda$ is close to $\lambda_1$, can be used to provide existence of solutions when $\lambda\in(\lambda_1,\lambda^*)$, however, it does not apply to the case $\lambda\ge\lambda^*$. In order to solve this problem, we borrow some ideas introduced in \cite{yavkay}.

Let us introduce our main assumptions. For a function $g:\mathbb{R}^N\to \mathbb{R}$, denote $\Omega_g^+=\{x\in \mathbb{R}^N :g(x)>0\}$, $\Omega_g^-=\{x\in \mathbb{R}^N :g(x)<0\}$ and $\Omega_g^0=\{x\in \mathbb{R}^N :g(x)=0\}$. For a bounded open set $U\subset \mathbb{R}^N$ we denote by $(\lambda_1(U),\phi_1(U))$ the first eigenpair associated with the operator $-\Delta_p$ over $U$, when it exists, for example when $U$ is a bounded regular domain.  We assume the following hypotheses on $h,f$:

\begin{description}
	\item[($F_1$)] $\Omega_f^+,\Omega_f^-$ are non empty sets with positive measure;
	\item[($F_2$)] if $\operatorname{int}(\Omega_f^0)\neq \emptyset$ then $\lambda_1(\operatorname{int}(\Omega_f^+\cup\Omega_f^0))<\lambda_1(\operatorname{int}(\Omega_f^0))$;
	\item[($F_\infty$)] $\lim_{|x|\to \infty}f(x)=f(\infty)<0$;
	\item[($F_{\phi_1}$)] $\int f|\phi_1|^\gamma<0$.
\end{description} 
\begin{remark} \noindent 
	\begin{enumerate}
		\item Hypotheses ($F_1$), ($F_\infty$) and ($F_{\phi_1}$) were all used in \cite{costa}.
		\item Hypothesis ($F_\infty$) implies that $\Omega_f^+,\Omega_f^0$ are bounded sets and hence the eingevalues that appear in ($F_2$) are well defined.
		\item Hypothesis ($F_2$) is the so-called ``thickness" hypothesis (in a more quantitative form). We need it here to show existence of solutions when $\lambda\ge\lambda^*$.
	\end{enumerate}
\end{remark}

In order to study the existence of solutions for problem~\eqref{p} we use an approach based on Nehari manifold method, see, e.g., \cite{drabek,yav1,yav2}. Associated to Problem \eqref{p} we have the so-called \textit{extreme value of the Nehari manifold method} which is defined by the following minimization problem 
\begin{equation*}
\lambda^*=\inf\left\{\frac{\int |\nabla u|^p}{\int h |u|^p}: \int f|u|^\gamma\ge 0,\ \int h|u|^p>0 \right\},
\end{equation*}
see \cite{yav2,ouyang}. The extreme value $\lambda^{*}$ defines a threshold for the applicability of the Nehari manifold method, in the sense that if $\lambda<\lambda^{*}$ then the Nehari set is a $C^{1}$-manifold and standard variational techniques may be applied in order to find critical points. In \cite{costa}, in order to show existence of two positive solutions for \eqref{p} when $\lambda$ is close to $\lambda_1$, the authors have used that the Nehari set is in fact a manifold which is not topologically connected. Hence a minimization argument in different components may be applied in order to find two positive solutions for Problem~\eqref{p}. Since by definition, whenever $\lambda\in(\lambda_{1},\lambda^{*})$ we have that the Nehari set is a manifold, it follows that the method used in \cite{costa} does work for all $\lambda\in (\lambda_1,\lambda^*)$. A natural question arises: Can the same result be obtained when $\lambda\geq\lambda^{*}$? In \cite{yavkay}, the authors have answered this question when the problem is defined on a bounded set. Precisely, they have proved that there exist at least two positive solutions for problem \eqref{pat}, provided that $\lambda\in(\lambda^{*},\Lambda)$, for some $\Lambda>\lambda^{*}$. For this purpose, the authors have used a variant of Nehari manifold method. Our main goal here is to answer the question when $\Omega=\mathbb{R}^{N}$. Due to the lack os compactness, it is necessary to introduce new techniques in the method and assumption $(F_{\infty})$ plays a very important role in the proof.

Now we are ready to state our main result.

\begin{theorem}\label{thm} 
	Suppose that $(F_1)$, $(F_2)$, $(F_{\infty})$ and $(F_{\phi_{1}})$ hold. Then, $\lambda^*>\lambda_1$ and there exists $\varepsilon>0$ such that problem \eqref{p} has at least two positive solutions for all $\lambda\in (\lambda_1,\lambda^*+\varepsilon)$.
\end{theorem}

\begin{remark} If we define 
	
	\begin{equation*}
	\lambda^{-*}=\sup\left\{\frac{\int |\nabla u|^p}{\int h |u|^p}: \int f|u|^\gamma=0,\ \int h|u|^p<0 \right\},
	\end{equation*}
	then a similar theorem can be proved in the case that $\lambda^{-*}$ has a minimizer and $\lambda<0$.
\end{remark}
%

The paper is organized as follows: In the forthcoming Section we introduce and study the Nehari sets associated to our problem. In Section~\ref{s3}, we show the existence of two positive solutions to Problem~\eqref{p}, for $\lambda\in(\lambda_{1},\lambda^{*}]$. In Section~\ref{s4}, we prove the existence of one positive solution to Problem~\eqref{p} when $\lambda>\lambda^{*}$. In Section~\ref{s5}, we use a Mountain Pass type argument to obtain the second positive solution when $\lambda>\lambda^{*}$, which concludes the proof of Theorem~\ref{thm}. Throughout the paper, we assume that all the hypotheses of Theorem~\ref{thm} hold.

\section{Finer Estimates Over the Nehari Sets}

In this Section we study the so called Nehari sets. In what follows, we use the notation
 \[
  H_{\lambda}(u)=\int|\nabla u|^{p}-\lambda\int h|u|^{p} \quad \mbox{and} \quad F(u)=\int f|u|^{\gamma}, \quad u\in E.
 \]
For each $\lambda\in \mathbb{R}$, the energy functional associated with problem \eqref{p} is given by
\begin{equation*}
\Phi_{\lambda}(u)=\frac{1}{p}H_{\lambda}(u)-\frac{1}{\gamma}F(u),\quad u\in E.
\end{equation*}
We say that $u\in E$ is a solution to \eqref{p} if $u$ is a critical point of the $C^1$ functional $\Phi_\lambda$. The Nehari set associated to $\Phi_\lambda$ is defined as 
\begin{equation*}
\mathcal{N}_{\lambda}:=\left\{u\in E\backslash\{0 \}: \Phi'_{\lambda}(u)(u)=0 \right\}.
\end{equation*}
Observe that if $u\in E$ is a nontrivial critical point of $\Phi_\lambda$, then $u\in \mathcal{N}_\lambda$. We split $\mathcal{N}_{\lambda}$ into three disjoint subsets:
\begin{equation*}
\mathcal{N}_{\lambda}^{+}:=\left\{u\in\mathcal{N}_{\lambda}:\Phi''_{\lambda}(u)(u,u)>0 \right\}=\{u\in\mathcal{N}_{\lambda}:H_\lambda(u)<0,\ F(u)<0\},
\end{equation*}
\begin{equation*}
\mathcal{N}_{\lambda}^{-}:=\left\{u\in\mathcal{N}_{\lambda}:\Phi''_{\lambda}(u)(u,u)<0 \right\}=\{u\in\mathcal{N}_{\lambda}:H_\lambda(u)>0,\ F(u)>0\},
\end{equation*}
 \begin{equation*}
\mathcal{N}_{\lambda}^{0}:=\left\{u\in\mathcal{N}_{\lambda}:\Phi''_{\lambda}(u)(u,u)=0 \right\}=\{u\in E\setminus\{0\}:H_\lambda(u)=0,\ F(u)=0\}.
\end{equation*}  
By using the Implicit Function Theorem, one can easily prove the following result:
\begin{lemma} If $\mathcal{N}_{\lambda}^{+},\mathcal{N}_{\lambda}^{-}$ are non-empty, then $\mathcal{N}_{\lambda}^{+},\mathcal{N}_{\lambda}^{-}$ are $C^1$ manifolds of codimension $1$ in $E$. 
\end{lemma} 
The main point in defining the Nehari manifolds $\mathcal{N}_{\lambda}^{+},\mathcal{N}_{\lambda}^{-}$ is that $\mathcal{N}_\lambda^+\cup\mathcal{N}_\lambda^-$ is a natural constraint to our problem as we see in the next proposition. 
\begin{proposition}\label{critialpoints}
	If $u\in \mathcal{N}_\lambda^+\cup\mathcal{N}_\lambda^-$ is a critical point of $\Phi_\lambda$ restricted to $\mathcal{N}_\lambda^+\cup\mathcal{N}_\lambda^-$, then $u$ is a critical point of $\Phi_\lambda$.
\end{proposition}
In general the Nehari set $\mathcal{N}_\lambda^0$ is not a manifold. Thus, since $\mathcal{N}_\lambda^0\neq\emptyset$ implies that $\overline{\mathcal{N}_\lambda^+\cup\mathcal{N}_\lambda^-}\cap \mathcal{N}_\lambda^0\neq\emptyset$, we must take some care with the set $\mathcal{N}_\lambda^0$ in order to search for critical points of $\Phi_\lambda$ restricted to $\mathcal{N}_\lambda^+\cup\mathcal{N}_\lambda^-$. The study of $\mathcal{N}_\lambda^0$ is related to the extreme value (see \cite{yav2})
\begin{equation}\label{lambdaex}
\lambda^*=\inf\left\{\frac{\int |\nabla u|^p}{\int h |u|^p}: \int f|u|^\gamma\ge 0,\ \int h|u|^p>0 \right\}.
\end{equation} 
Throughout the paper, we eventually study the convergence of minimizing sequences. For this purpose we introduce some notations. Let $(u_{n})\subset E$ be a sequence such that $u_{n}\rightharpoonup u$ weakly in $E$. Following \cite{costa} we define
\begin{equation}\label{alpha}
\alpha_{\infty}:=\lim_{R\rightarrow+\infty}\limsup_{n\rightarrow+\infty}\int_{\{|x|>R\}}|u_{n}|^{\gamma},
\end{equation}
\begin{equation}\label{beta}
\beta_{\infty}:=\lim_{R\rightarrow+\infty}\limsup_{n\rightarrow+\infty}\int_{\{|x|>R\}}|\nabla u_{n}|^{p},
\end{equation}
where $\{|x|>R\}=\{x\in\mathbb{R}^{N}:|x|>R\}$. Hence, one has
\begin{equation}\label{alpha1}
\int |u|^{\gamma}+\alpha_{\infty}=\limsup_{n\rightarrow+\infty}\int |u_{n}|^{\gamma},
\end{equation}
\begin{equation}\label{alpha2}
\int f|u|^{\gamma}+\alpha_{\infty}f(\infty)=\limsup_{n\rightarrow+\infty}\int f|u_{n}|^{\gamma},
\end{equation}
\begin{equation}\label{beta1}
\int |\nabla u|^{p}+\beta_{\infty}=\limsup_{n\rightarrow+\infty}\int |\nabla u_{n}|^{p}.
\end{equation}
\begin{lemma}\label{l1} There holds $\lambda^*>\lambda_1$. Moreover, there exists a nonnegative function $u^{*}\in E$ such that
	
	\begin{equation*}
	\lambda^{*}=\frac{\int |\nabla u^{*}|^p}{\int h |u^{*}|^p} \quad \mbox{and} \quad \int f|u^{*}|^{\gamma}=0.
	\end{equation*}
	
\end{lemma}
\begin{proof} Let $(u_{n})\subset E$ be a normalized minimizing sequence to $\lambda^{*}$, that is, $\|u_{n}\|=1$ and
	
	\begin{equation*}
	\lim_{n\rightarrow+\infty}\frac{\int |\nabla u_{n}|^{p}}{\int h |u_{n}|^p}=\lambda^{*}, \quad \int f|u_{n}|^{\gamma}\geq0, \quad \int h|u_{n}|^{p}>0.
	\end{equation*}
		Notice that $u_{n}\rightharpoonup u^{*}$ weakly in $E$ and $\|u_{n}\|_{D^{1,p}}\rightarrow A\geq0$ as $n\rightarrow+\infty$. If $A=0$, then $u^{*}=0$. Thus, it follows from \eqref{alpha1} that $\alpha_{\infty}=1$. Hence, in view of \eqref{alpha2} we have
	\begin{equation*}
	f(\infty)=\limsup_{n\rightarrow+\infty}\int f|u_{n}|^{\gamma}\geq0,
	\end{equation*}	
	which contradicts assumption $(F_\infty)$. Therefore, $A>0$ and 	
	\begin{equation*}
	\lim_{n\rightarrow+\infty}\frac{A}{\int h |u_{n}|^p}=\lambda^{*}, \quad \int f|u_{n}|^{\gamma}\geq0, \quad \int h|u_{n}|^{p}>0.
	\end{equation*}	
	Now, we claim that $u^{*}\neq0$. In fact, if $u^{*}=0$, then by compactness (see \cite{costa}) we have that $\lim_{n\rightarrow+\infty}h|u_{n}|^{p}=0$, which is not possible and therefore $u^*\neq 0$. From $(F_\infty)$ and \eqref{alpha2} one has	
	\begin{equation*}
	\int f|u^*|^{\gamma}=-\alpha_{\infty}f(\infty)+\limsup\int f|u_{n}|^{\gamma}\geq0.
	\end{equation*}		
	Thus, we conclude that	
	\begin{equation*}
	0<\|u^{*}\|\leq 1, \quad \int h|u^{*}|^{p}>0 \quad \mbox{and} \quad \int f|u^{*}|^{\gamma}\geq0.
	\end{equation*}	 	
	We claim that $\lambda^{*}$ is achieved by $u^{*}$. For this purpose, it is suffices to prove that $u_{n}\rightarrow u^{*}$ strongly in $D^{1,p}(\mathbb{R}^{N})$. Suppose by contradiction that the strong convergence does not hold. Thus, $\|\nabla u^{*}\|_p<\liminf_{n\rightarrow+\infty}\|\nabla u_{n}\|_p$. Hence, we deduce that	
	\begin{equation*}
	\frac{\int |\nabla u^*|^p}{\int h |u^{*}|^p}<\liminf_{n\rightarrow+\infty}\frac{\int |\nabla u_n|^p}{\int h |u_{n}|^p}=\lambda^{*},
	\end{equation*}	
	which contradicts the definition of $\lambda^{*}$. Therefore, $u^*$ is a minimizer of $\lambda^{*}$. 
	
Now we claim that $\lambda^*>\lambda_1$. Indeed, it is obvious that $\lambda^*\ge \lambda_1$. If $\lambda^*=\lambda_1$, then $u^*=\phi_1$ which contradicts the hypothesis (F$_{\phi_1}$). Therefore $\lambda^*>\lambda_1$.

	It remains to prove that $\int f|u^{*}|^{\gamma}=0$. Suppose by contradiction that
	$
	\int f|u^{*}|^{\gamma}>0.
	$	
	Thus, the set	
	\begin{equation*}
	\left\{u\in E: \int f|u|^{\gamma}>0 \mbox{ and } \int h|u|^{p}>0 \right\},
	\end{equation*}	
	is an open subset of $E$. Taking into account that $u^{*}$ is a local minimum, one sees that	
	\begin{equation*}
	p\int |\nabla u^{*}|^{p-2}\nabla u^{*}\nabla v\int h|u^{*}|^{p}-p\int |\nabla u^{*}|^{p}\int h|u^{*}|^{p-2}u^{*}v=0,
	\end{equation*}	
	for all $v\in E$. Since $u^{*}$ is a minimizer of $\lambda^{*}$ we conclude that	
	\begin{equation*}
	\int |\nabla u^{*}|^{p-2}\nabla u^{*}\nabla v-\lambda^{*}\int h|u^{*}|^{p-2}u^{*}v=0,\quad \forall\ v\in E.
	\end{equation*}	
	Once $E$ is dense in $D^{1,p}(\mathbb{R}^{N})$ and the functional $$D^{1,p}(\mathbb{R}^{N})\ni v\longmapsto \int h|u^*|^{p-2}u^*v$$ is completely continuous, we conclude that 	
	\begin{equation*}
	\int |\nabla u^{*}|^{p-2}\nabla u^{*}\nabla v-\lambda^{*}\int h|u^{*}|^{p-2}u^{*}v=0,\quad \forall\ v\in D^{1,p}(\mathbb{R}^{N}).
	\end{equation*}	
	Thus, $\lambda^{*}$ is an eigenvalue of Problem \eqref{pe}. Recall that the unique eigenfunction which does not change sign is the one associated to $\lambda_{1}$. Since $\lambda^{*}>\lambda_{1}$ and $u^{*}$ is non-negative, we get a contradiction and therefore $\int f|u^{*}|^{\gamma}=0$.
\end{proof}
In view of the preceding Lemma, we obtain the following Corollary: 
\begin{corollary}\label{cor1} There holds	
	\begin{equation*}
	\lambda^*=\inf\left\{\frac{\int |\nabla u|^p}{\int h |u|^p}: \int f|u|^\gamma= 0,\ \int h|u|^p>0 \right\}.
	\end{equation*}	
\end{corollary}
In light of Lemma \ref{l1}, we know that the minimization problem $\lambda^*$ has a minimizer. The next result ensures that we can use this minimizer to get a solution of Problem \eqref{p}.
\begin{lemma}\label{lagranextre}
	Suppose that $u^{*}\in E$ is a minimizer of $\lambda^*$. Then, there exists a constant $t_{0}>0$ such that $t_{0}u^{*}$ is a solution of Problem \eqref{p} with $\lambda=\lambda^*$. Moreover $t_0u^*\in \mathcal{N}_{\lambda^*}^0$.
\end{lemma}
\begin{proof}
	 Let $u^{*}$ be a minimizer of $\lambda^*$. In order to use the Lagrange Multiplier Theorem, we first prove that the derivative of the function $D^{1,p}\ni u\mapsto G(u)\equiv(\int h |u|^p,\int f|u|^\gamma)$ is surjective at $u^*$. In fact, let $\alpha,\beta\in\mathbb{R}$ be such that
	\begin{equation}
	\alpha p\int h|u^{*}|^{p-2}u^{*}v+\beta\gamma\int f|u^{*}|^{\gamma-2}u^{*}v=0, \quad \forall \ v\in E.
	\end{equation}
	By taking $v=u^{*}$ we easily conclude that $\alpha=0$. Now, let us suppose by contradiction that $\beta\neq0$. Thus, we have	
	\begin{equation*}
	\int f|u^{*}|^{\gamma-2}u^{*}v=0, \quad \forall \ v\in E,
	\end{equation*}	
	which implies that $f|u^{*}|^{\gamma-2}u^{*}=0$ a.e. in $\mathbb{R}^{N}$ and supp$(u^{*})\subset\Omega^{0}$. If int$(\Omega^{0}_{f})=\emptyset$, then we get a contradiction. Let int$(\Omega^{0}_{f})\neq\emptyset$ and consider	
	\begin{equation*}
	\lambda_{1}(\Omega^{0}_{f}\cup\Omega^{+}_{f})=\inf\left\{\frac{\int|\nabla u|^{p}}{\int h|u|^{p}}: u\in W_0^{1,p}(\Omega^{0}\cup \Omega^{+}), \ \int h|u|^{p}>0 \right\}.
	\end{equation*}	
	By using $(F_{2})$ and the fact that supp$(u^{*})\subset\Omega^{0}$ we deduce that	
	\begin{equation*}
	\lambda^{*}\leq \lambda_{1}(\Omega^{0}_{f}\cup\Omega^{+}_{f})<\lambda_{1}(\Omega^{0}_{f})=\lambda^{*},
	\end{equation*}	
	which is not possible. Therefore, $G'(u^*)$ is surjective and from the Lagrange Multiplier Theorem, there exists $\nu\ge 0$ such that
	\begin{equation}\label{kj1}
	\left(\frac{\int|\nabla u^{*}|^{p}}{\int h|u^{*}|^{p}} \right)'\cdot v=\nu\gamma\int f|u^{*}|^{\gamma-2}u^{*}v,\quad \forall \ v\in E.
	\end{equation}
	Notice that
	\begin{equation}\label{kj2}
	\left(\frac{\int|\nabla u^{*}|^{p}}{\int h|u^{*}|^{p}} \right)'\cdot v=\frac{p}{\int h|u^{*}|^{p}}\left(\int |\nabla u^{*}|^{p-2}\nabla u^{*}\nabla v-\lambda^{*}\int h|u^{*}|^{p-2}u^{*}v \right).
	\end{equation}
	We claim that $\nu\neq0$. Indeed, if $\nu=0$, we combine \eqref{kj1} with \eqref{kj2} to conclude as in the proof of Lemma \ref{l1} that $\lambda^*=\lambda_1$, which is a contradiction. Therefore	$\nu\neq 0$ and 
	\begin{equation*}
	p\int |\nabla u^{*}|^{p-2}\nabla u^{*}\nabla v-p\lambda^{*}\int h|u^{*}|^{p-2}u^{*}v=\gamma\nu\int h|u^*|^p\int f|u^{*}|^{\gamma-2}u^{*}v, \quad \forall\ v\in E.
	\end{equation*}	
	Multiplying the last equation by $|t|^{p-2}t$, where $t\neq 0$, we obtain that for all $v\in E$, there holds
	\begin{equation*}
	p\int |\nabla (tu^{*})|^{p-2}\nabla tu^{*}\nabla v-p\lambda^{*}\int h|tu^{*}|^{p-2}tu^{*}v=\gamma\nu|t|^{p-\gamma}\int h|u^*|^p\int f|tu^{*}|^{\gamma-2}u^{*}v.
	\end{equation*}	
	By choosing $t=t_0$ such that $\gamma\nu|t_0|^{p-\gamma}\int h|u^*|^p=1$, the proof is completed.
\end{proof}
As a consequence of Corollary \ref{cor1} and Lemma \ref{lagranextre} we obtain our main result in relation to the Nehari set $\mathcal{N}_\lambda^0$.
\begin{proposition}\label{neharizero} If $\lambda\in(\lambda_1,\lambda^*)$, then $\mathcal{N}_\lambda^0=\emptyset$. Moreover, if $\lambda\ge \lambda^*$, then $\mathcal{N}_\lambda^0\neq\emptyset$.
\end{proposition}
\begin{proof} Indeed, assume that $\lambda\in (\lambda_1,\lambda^*)$ and suppose on the contrary that there exists $u\in \mathcal{N}_\lambda^0$. From the definition we have that $u\neq 0$ and $u\in\mathcal{N}_\lambda^0$ if, and only if
		\begin{equation*}
	\int |\nabla u|^p-\lambda\int h|u|^p=\int f|u|^\gamma=0.
	\end{equation*}
It follows that 
\begin{equation*}
\frac{\int |\nabla u|^p}{\int h|u|^p}=\lambda<\lambda^*,\ \int h|u|^p>0,\ \int f|u|^\gamma=0,
\end{equation*}
which contradicts Corollary \ref{cor1} and therefore  $\mathcal{N}_\lambda^0=\emptyset$ if $\lambda\in (\lambda_1,\lambda^*)$. In view of \ref{lagranextre} we conclude that $\mathcal{N}_{\lambda^*}^0\neq \emptyset$. In order to prove that $\mathcal{N}_{\lambda}^0\neq \emptyset$ for $\lambda>\lambda^*$, we note that the functional 
$R:\mathcal{X}\to [0,\infty)$ defined by
\begin{equation*}
R(u):=\frac{\int |\nabla u|^p}{\int h|u|^p},\quad u\in \mathcal{X},
\end{equation*} 
where $\mathcal{X}=\{ u\in E\setminus\{0\}:\ \int h|u|^p>0,\ \int f|u|^\gamma=0\}$ is continuous. Note that if $\lambda=R(u)$, then 
\begin{equation*}
\int |\nabla u|^p-\lambda\int h|u|^p=\int f|u|^\gamma=0,
\end{equation*}
and hence $u\in \mathcal{N}_\lambda^*$. Therefore it is enough to prove that there exists a sequence $(u_n)\subset X$ such that $R(u_n)\to \infty$ as $n\to \infty$. For this purpose note that if 
\begin{equation*}
u\in\overline{\mathcal{X}\cap\{u\in E:\ \|u\|=1\}},
\end{equation*} 
and $u_n \to u$ in $E$, then $R(u_n)\to \infty$. Since $\int h|tu|^p=t^p\int h|u|^p$ and $\int f|tu|^\gamma=t^\gamma\int h|u|^\gamma$ one conclude that $\overline{\mathcal{X}\cap\{u\in E:\ \|u\|=1\}}\neq\emptyset$ and therefore the proof is completed.
\end{proof}
\section{Two Solutions For $\lambda\in(\lambda_1,\lambda^*]$}\label{s3}
In this section we show the existence of two positive solutions to Problem \eqref{p} for $\lambda\in (\lambda_1,\lambda^*]$. We point out that in \cite{costa} the existence of two positive solutions was proved for $\lambda>\lambda_1$ and close to $\lambda_1$. However, we emphasize here that the method employed there does work for all $\lambda\in(\lambda_1,\lambda^*)$. The case $\lambda=\lambda^*$ is more delicate and requires new ideas. Consider the constrained minimization problems
\begin{equation}\label{mini-}
\hat{J}_\lambda^-:=\inf \{\Phi_\lambda(u),\ \forall\ u\in \mathcal{N}_{\lambda}^{-}\},
\end{equation}
and
\begin{equation}\label{mini+}
\hat{J}_\lambda^+:=\inf \{\Phi_\lambda(u),\ \forall\ u\in \mathcal{N}_{\lambda}^{+}\}.
\end{equation}
Similarly to \cite{costa} we introduce the following sets:
\begin{equation*}
L^{-}(\lambda):=\left\{u\in E: \|u\|_{D^{1,p}}=1 \mbox{ and } \int |\nabla u|^{p}-\lambda\int h|u|^{p}<0 \right\},
\end{equation*}
\begin{equation*}
B^{+}(\lambda):=\left\{u\in E: \|u\|_{D^{1,p}}=1 \mbox{ and } \int f|u|^{\gamma}>0 \right\}.
\end{equation*}
As an application of Proposition \ref{neharizero} we obtain the following Corollary:
\begin{corollary}\label{cor} For each $\lambda\in (\lambda_1,\lambda^*)$, there holds $\overline{L^-(\lambda)}\cap \overline{B^+(\lambda)}=\emptyset$. For each $\lambda\ge \lambda^*$, there holds $\overline{L^-(\lambda)}\cap \overline{B^+(\lambda)}\neq\emptyset$.
\end{corollary}
\begin{proof} Indeed, suppose that there exists $u\in \overline{L^-(\lambda)}\cap \overline{B^+(\lambda)}$, therefore $u\in \mathcal{N}_\lambda^0$ and from Proposition \ref{neharizero} $\lambda\ge \lambda^*$.
	\end{proof}
 By using Corollary~\ref{cor}, J. Chabrowski and D.G. Costa \cite{costa}, have proved the following Theorem:
\begin{theorem}\label{exisbeforlam} For each $\lambda\in (\lambda_1,\lambda^*)$, there exists $u_\lambda\in \mathcal{N}_\lambda^+$ and $w_\lambda\in \mathcal{N}_\lambda^-$ such that $\hat{J}_\lambda^-=\Phi_\lambda(w_\lambda)$, $\hat{J}_\lambda^+=\Phi_\lambda(u_\lambda)$ and $u_\lambda,w_\lambda$ are solutions of \eqref{p}.
\end{theorem}
Since  $\overline{L^-(\lambda)}\cap \overline{B^+(\lambda)}\neq\emptyset$ for $\lambda\ge\lambda^*$ the technique used in \cite{costa} no longer applies to prove existence of solutions, therefore, a new idea has to be introduced in order to study this case. We start with the case $\lambda=\lambda^*$. Let us introduce the subset of $E$ given by
\begin{equation*}
\Theta_{\lambda}^{+}:=\left\{u\in E\backslash\{0 \}:\int|\nabla u|^{p}-\lambda\int h|u|^{p}<0 \mbox{ and } \int f|u|^{\gamma}<0 \right\}.
\end{equation*}
By using the Fibering Method of Pohozaev \cite{yav1,pohozaev}, we have that for each $u\in \Theta_{\lambda}^{+}$, there exists a unique $s=s_{\lambda}^{+}(u)>0$ given by 
\begin{equation}\label{kj22}
s_{\lambda}^{+}(u)=\left[\frac{\int|\nabla u|^{p}-\lambda\int h|u|^{p}}{\int f|u|^{\gamma}} \right]^{\frac{1}{\gamma-p}},
\end{equation} 
such that $su\in\mathcal{N}_{\lambda}^{+}$. Hence, we have the following characterization
\begin{equation*}
\mathcal{N}_{\lambda}^{+}=\left\{su:s=s_{\lambda}^{+}(u) \mbox{ and } u\in\Theta_{\lambda}^{+} \right\}.
\end{equation*}
With the preceeding parametrization of the Nehari manifold we observe that  $J_{\lambda}^{+}:=\Phi_{\lambda}\mid_{\mathcal{N}_{\lambda}^{+}}:\mathcal{N}_{\lambda}^{+}\rightarrow\mathbb{R}$ is given by
\begin{equation}\label{kj3}
J_{\lambda}^{+}(u)=:\Phi_{\lambda}(s_{\lambda}^{+}(u)u)=-c_{p,\gamma}\frac{\left|\int|\nabla u|^{p}-\lambda\int h|u|^{p} \right|^{\frac{\gamma}{\gamma-p}}}{\left|\int f|u|^{\gamma} \right|^{\frac{p}{\gamma-p}}}, \quad u\in\Theta_{\lambda}^{+},
\end{equation}
where $c_{p,\gamma}=(\gamma-p)/p\gamma$. 

\begin{remark}\label{Jzeroh}Notice that $J_{\lambda}^{+}(u)$ is $0$-homogeneous on $\Theta_{\lambda}^{+}$, i.e., $J_{\lambda}^{+}(tu)=J_{\lambda}^{+}(u)$, for each $t>0$ and $u\in\Theta_{\lambda}^{+}$. For this reason, throughout the paper we consider normalized sequences. 
\end{remark}

By similar ideas used in \cite{yav1} we deduce the following technical Lemmas:

\begin{lemma}\label{yavdat}
	If $D_{v}J_{\lambda}^{+}(v)(\eta)=0$ for all $\eta\in E\backslash\{0\}$, then $s_{\lambda}^{+}(v)v$ is a weak solution of Problem \eqref{p}.
\end{lemma}
\begin{lemma}\label{de} The function $(\lambda_1,\lambda^*)\ni\lambda\mapsto\hat{J}_{\lambda}^{+}$ is decreasing.
\end{lemma}
\begin{proof} Suppose that $\lambda<\lambda'$ and observe that $\Theta^+_\lambda\subset \Theta^+_{\lambda'}$. Choose $v_\lambda\in \Theta^+_\lambda$ (given by Theorem \ref{exisbeforlam}) such that $\hat{J}_\lambda=J_\lambda(v_\lambda)$. It follows that 
	\begin{equation*}
	\hat{J}_{\lambda'}\le J_{\lambda'}(v_\lambda)<J_\lambda(v_\lambda)=\hat{J}_\lambda,
	\end{equation*}
	which finishes the proof.
\end{proof}

\begin{lemma}\label{liminf}
	Under the assumptions of the Lemma \ref{l3} there holds
	\[
	J_{\lambda}(v)\leq \liminf_{n\rightarrow+\infty}J_{\lambda_{n}}(v_{n}).
	\]
\end{lemma}
\begin{proof}
	In view of \eqref{alpha2} we have that $F(v)\geq \limsup_{n\rightarrow+\infty} F(v_{n})$. Thus, $-F(v)\leq \liminf_{n\rightarrow+\infty}(-F(v_{n}))$. Since $F(v),F(v_{n})<0$ one has $|F(v)|\leq \liminf_{n\rightarrow+\infty}|F(v_{n})|$. Hence, we have
	\begin{equation}\label{kj13}
	-c_{\gamma,p}\frac{|H_{\lambda}(v)|^{\frac{\gamma}{\gamma-p}}}{|F(v)|^{\frac{p}{\gamma-p}}}\leq -c_{\gamma,p}\frac{|H_{\lambda}(v)|^{\frac{\gamma}{\gamma-p}}}{\displaystyle\liminf_{n\rightarrow+\infty}|F(v_{n})|^{\frac{p}{\gamma-p}}}.
	\end{equation}
	Notice that
	\[
	-c_{\gamma,p}|H_{\lambda}(v)|^{\frac{\gamma}{\gamma-p}}\leq -c_{\gamma,p}\limsup_{n\rightarrow+\infty}|H_{\lambda_{n}}(v_{n})|^{\frac{\gamma}{\gamma-p}}, 
	\]
	which together with \eqref{kj13} implies that
	\begin{equation}
		J_{\lambda}(v)  \leq  -c_{\gamma,p}\frac{\displaystyle\limsup_{n\rightarrow+\infty}|H_{\lambda_{n}}(v_{n})|^{\frac{\gamma}{\gamma-p}}}{\displaystyle\liminf_{n\rightarrow+\infty}|F(v_{n})|^{\frac{p}{\gamma-p}}}
		 =  \liminf_{n\rightarrow+\infty}J_{\lambda_{n}}(v_{n}),
	\end{equation}
	and the proof is complete.
\end{proof}

Now, we consider the minimization problem
\begin{equation*}
\hat{J}_{\lambda}^{+}:=\min\left\{J_{\lambda}^{+}(v):v\in\Theta_{\lambda}^{+}\cap S \right\},
\end{equation*}
where $S:=\{u\in E: \|u\|=1 \}$.
\begin{theorem}\label{l3}
	Suppose the assumptions of Theorem~\ref{thm}, then there exists a minimizer $v_{\lambda^{*}}\in\Theta_{\lambda^{*}}^{+}$ of $\hat{J}_{\lambda^{*}}^{+}$ such that $u_{\lambda^{*}}:=s(v_{\lambda^{*}})v_{\lambda^{*}}$ is a weak solution to \eqref{p}.
 
\end{theorem}
\begin{proof}
	Indeed suppose that $\lambda_{n}\uparrow\lambda^{*}$ as $n\rightarrow+\infty$. In light of Theorem \ref{exisbeforlam}, for each $n\in\mathbb{N}$ there exists $v_{\lambda_{n}}\in\Theta_{\lambda_{n}}^{+}\cap S$ such that $J_{\lambda_{n}}^{+}(v_{\lambda_{n}})=\hat{J}_{\lambda_{n}}^{+}$ and 
	\begin{equation}\label{kj4}
	-\Delta_p v_{\lambda_{n}} -\lambda_{n} h(x)|v_{\lambda_{n}}|^{p-2}v_{\lambda_{n}} = s_{\lambda_{n}}^{+}(v_{\lambda_{n}})^{\gamma-p}f(x)|v_{\lambda_{n}}|^{\gamma-2}v_{\lambda_{n}}.
	\end{equation}
	Once $|v_{\lambda_n}|$ also satisfies $J_{\lambda_{n}}^{+}(|v_{\lambda_{n}}|)=\hat{J}_{\lambda_{n}}^{+}$, we can assume without loss of generality that $v_{\lambda_n}\ge 0$ for each $n$. Up to a subsequence, we may assume that $v_{\lambda_{n}}\rightharpoonup v$ weakly in $E$. Arguing as in the proof of Lemma~\ref{l1} one can deduce that $\|v_{\lambda_{n}}\|_{D^{1,p}}\rightarrow B>0$ as $n\rightarrow+\infty$. Thus, since $v_{\lambda_{n}}\in\Theta_{\lambda_{n}}^{+}$ it follows that 
	\[
	0<B\leq \lambda^{*}\int h|v|^{p},
	\]
	which implies that $v\neq0$.
	
	Now, let us prove that $v\in\Theta_{\lambda^{*}}^{+}$. Since $v_{n}\in\Theta_{\lambda_{n}}^{+}$ one has
	\begin{equation}\label{1}
	\int|\nabla v|^{p}-\lambda^{*}\int h|v|^{p}\leq \liminf_{n\rightarrow+\infty}\left(\int|\nabla v_{\lambda_{n}}|^{p}-\lambda_{n}\int h|v_{\lambda_{n}}|^{p}\right)\leq0.
	\end{equation}
	and from the definition of $\lambda^{*}$ we also have that
	\begin{equation}\label{2}
	\int f|v|^{\gamma}\leq  0.
	\end{equation}	
	Now, we suppose by contradiction that
	$
	\int f|v|^{\gamma}=0.
	$
	From Corollary \ref{cor1}, we conclude that $H_{\lambda^*}(v)=0$, that is	
	\begin{equation*}
	\int|\nabla v|^{p}-\lambda^{*}\int h|v|^{p}=0,
\end{equation*}	
	which implies from \eqref{1} that
	\begin{equation}\label{kj5}
	\lim_{n\rightarrow+\infty}\left[\int|\nabla v_{\lambda_{n}}|^{p}-\lambda_{n}\int h|v_{\lambda_{n}}|^{p} \right]=0.
	\end{equation}
	By using \eqref{kj3}, the fact that the function $(\lambda_1,\lambda^*)\ni\lambda\mapsto\hat{J}_{\lambda}^{+}$ is decreasing (see Lemma \ref{de}) and \eqref{kj5}, we conclude that
	\begin{equation}\label{kj6}
	\liminf_{n\rightarrow+\infty}\int f|v_{\lambda_{n}}|^{\gamma}=0.
	\end{equation}
	Thus, using \eqref{alpha2} we deduce that $\alpha_{\infty}=0$. Therefore, $v_{\lambda_{n}}\rightarrow v$ strongly in $L^{\gamma}(\mathbb{R}^{N})$ which jointly with \eqref{kj5} implies that $v_{\lambda_{n}}\rightarrow v$ strongly in $E$. In view of \eqref{kj4}, it follows that 
	\begin{equation}\label{lim}
	-\Delta_p v -\lambda^{*} h(x)|v|^{p-2}v = \lim_{n\rightarrow+\infty}s_{\lambda_{n}}^{+}(v_{\lambda_{n}})^{\gamma-p}f(x)|v_{\lambda_{n}}|^{\gamma-2}v_{\lambda_{n}}, \quad \mbox{in} \hspace{0,2cm} \mathbb{R}^{N}.
	\end{equation}
	In view of Lemma~\ref{lagranextre}, there exists $t_{0}>0$ such that $t_{0}v$ is a solution of Problem~\eqref{p}. Thus, one has
	 \begin{equation}\label{kj17}
	  -\Delta_{p}v-\lambda^{*}h|v|^{p-2}v=t^{\gamma-p}f|v|^{\gamma-2}v, \quad \mbox{in} \hspace{0,2cm} \mathbb{R}^{N}.
	 \end{equation}
	Combining \eqref{lim} and \eqref{kj17} we obtain
	 \[
	  \lim_{n\rightarrow+\infty}s_{\lambda_{n}}^{+}(v_{\lambda_{n}})^{\gamma-p}f(x)|v_{\lambda_{n}}|^{\gamma-2}v_{\lambda_{n}}=t^{\gamma-p}f|v|^{\gamma-2}v.
	 \]
	Since $v_{\lambda_{n}}\rightarrow v$ strongly in $L^{\gamma}(\mathbb{R}^{N})$ we conclude that $\lim_{n\rightarrow+\infty}s_{\lambda_{n}}^{+}(v_{\lambda_{n}})^{\gamma-p}=t^{\gamma-p}$. Therefore, $s_{\lambda^{*}}^{+}\in(0,+\infty)$. Hence, we observe from \eqref{kj22} and \eqref{kj3} that
	\[
	0>J_{\lambda_{n}}^{+}(v_{\lambda_{n}})=-c_{p,\gamma}s_{\lambda_{n}}^{+}(v_{\lambda_{n}})^{p}\left|\int|\nabla v_{\lambda_{n}}|^{p}-\lambda_{n}\int h|v_{\lambda_{n}}|^{p} \right|=o_{n}(1),
	\]
	which is not possible once $(\lambda_1,\lambda^*)\ni\lambda\mapsto\hat{J}_{\lambda}^{+}$ is decreasing. Therefore, $\int f|v|^{\gamma}<0$. 
	
	Finally, if 
	\[
	\int|\nabla v|^{p}-\lambda^{*}\int h|v|^{p}=0,
	\]
	then using \eqref{alpha2} and arguing as in \eqref{kj5} and \eqref{kj6} we deduce that
	\[
	0\leq -\alpha_{\infty}f(\infty)=\int f|v|^{\gamma}<0,
	\]
	which is not possible. Therefore, $v\in\Theta_{\lambda^{*}}^{+}$. In order to prove that $v_{\lambda_{n}}\rightarrow v$ strongly in $E$, suppose by contradiction that the strong convergence does not hold, thus from Lemma \ref{liminf} we obtain that $$J_{\lambda^*}^+(v)<\displaystyle\liminf_{n\to \infty}J^+_{\lambda_n}(v_n)=\displaystyle\liminf_{n\to \infty}\hat{J}^+_{\lambda_n}.$$ Observe that for sufficiently large $n$ there holds $v_{\lambda_n}\in \Theta_{\lambda_n}^+$. Moreover one can easily see that $J_{\lambda_n}^+(v)\to J_{\lambda^*}^+(v)$. It follows that 
	\begin{equation*}
	J_{\lambda^*}^+(v)<\liminf_{n\to \infty}\hat{J}^+_{\lambda_n}\le \lim_{n\to \infty}J_{\lambda_n}^+(v)\to J_{\lambda^*}^+(v),
	\end{equation*}
	which is a contradiction and therefore $v_{\lambda_{n}}\rightarrow v$ strongly in $E$. We conclude that
	\begin{equation}\label{kj10}
	\hat{J}_{\lambda_{n}}^{+}=J_{\lambda_{n}}^{+}(v_{\lambda_{n}})\rightarrow J_{\lambda^{*}}^{+}(v)\geq\hat{J}_{\lambda^{*}}^{+}, \quad \mbox{as} \hspace{0,2cm} \lambda_{n}\uparrow \lambda^{*}.
	\end{equation}
	We claim that $J_{\lambda^{*}}^{+}(v)=\hat{J}_{\lambda^{*}}^{+}$. Suppose by contradiction that $J_{\lambda^{*}}^{+}(v)>\hat{J}_{\lambda^{*}}^{+}$. Note that or $\hat{J}_{\lambda^{*}}^{+}=-\infty$ or $\hat{J}_{\lambda^{*}}^{+}\in(-\infty,0)$. Let us suppose the case $\hat{J}_{\lambda^{*}}^{+}=-\infty$, the other one is studied by a similar argument. In this case, for any $\eta>0$ there exists $w_{\eta}\in\Theta_{\lambda^{*}}^{+}$ such that $J_{\lambda^{*}}^{+}(w_{\eta})<J_{\lambda^{*}}^{+}(v)-\eta$. For given $\varepsilon>0$ there is $n_{1}\in\mathbb{N}$ such that
	\begin{equation}\label{kj11}
	|J_{\lambda_{n}}^{+}(w_{\eta})-J_{\lambda^{*}}^{+}(w_{\eta})|<\varepsilon, \quad \forall n\geq n_{1}.
	\end{equation}
	In view of \eqref{kj10} there exists $n_{2}\in\mathbb{N}$ such that
	\begin{equation}\label{kj12}
	|\hat{J}_{\lambda_{n}}^{+}-J_{\lambda^{*}}^{+}(v)|<\varepsilon, \quad \forall n\geq n_{2}.
	\end{equation}
	Thus, for $n\geq n_{0}:=\max\{n_{1},n_{2}\}$, it follows from \eqref{kj11} and \eqref{kj12} that
	\[
	J_{\lambda^{*}}^{+}(v)-\varepsilon<\hat{J}_{\lambda_{n}}^{+}\leq J_{\lambda_{n}}^{+}(w_{\eta})<J_{\lambda^{*}}^{+}(w_{\eta})+\varepsilon<J_{\lambda^{*}}^{+}(v)-\eta+\varepsilon.
	\]
	Since $\varepsilon$ and $\eta$ are arbitrary we get a contradiction. Therefore, $J_{\lambda^{*}}^{+}(v)=\hat{J}_{\lambda^{*}}^{+}$ and if $u_{\lambda^*}:=s_{\lambda^*}(v)v$, then from Lemma \ref{yavdat} the proof is complete.
	
\end{proof}

\section{First Solution for $\lambda>\lambda^*$}\label{s4}

In this section we prove existence of one positive solution to problem \eqref{p} when $\lambda>\lambda^*$. To this end we need to provide some estimates concerning the minimizers of $\hat{J}_{\lambda^*}^+$.
\begin{lemma}\label{separataion}
	There exists $\mu_{0}\in(\lambda_{1},\lambda^{*})$ such that each minimizer $v_{\lambda^{*}}\in\Theta_{\lambda^{*}}^{+}$ of $\hat{J}_{\lambda^{*}}^{+}$ satisfies
	\[
	\int |\nabla v_{\lambda^{*}}|^{p}-\mu_{0}\int h|v_{\lambda^{*}}|^{p}<0.
	\]
\end{lemma}
\begin{proof}
	Suppose by contradiction that for each $\mu\in (\lambda_1,\lambda^*)$, there exists $v_\mu\in \Theta_{\lambda^{*}}^{+}\cap S$ such that 
	\begin{equation}\label{tt}
	\hat{J}_{\lambda^{*}}^{+}=J_{\lambda^{*}}^{+}(v_\mu) \quad \mbox{and} \quad \int |\nabla v_\mu|^{p}-\mu\int h|v_\mu|^{p} \ge 0.
	\end{equation}
	It follows that there exist sequences $\mu_n\uparrow \lambda^*$ and $v_n\equiv v_{\mu_n}\in \Theta_{\lambda^{*}}^{+}\cap S$ satisfying \eqref{tt}. Observe that   
	\[
	 |H_{\lambda^*}(v_n)-H_{\mu_n}(v_n)|=\left|(\lambda^*-\mu_n)\int h|v_n|^p\right|\to 0, \quad as \hspace{0,2cm} n\to \infty.
	\]
    Therefore, $H_{\lambda^*}(v_n)\to 0$ as $n\to \infty$. We may assume, up to a subsequence, that $v_{n}\rightharpoonup v$ weakly in $E$. Arguing as in the proof of Lemma~\ref{l3} we conclude that $v\neq0$. Notice that
	\[
	0>\hat{J}_{\lambda^{*}}^{+}=J_{\lambda^{*}}^{+}(v_{n})=-c_{p,\gamma}\frac{\left|\int|\nabla v_{n}|^{p}-\lambda^{*}\int h|v_{n}|^{p} \right|^{\frac{\gamma}{\gamma-p}}}{\left|\int f|v_{n}|^{\gamma} \right|^{\frac{p}{\gamma-p}}},
	\]
	which implies that
	\begin{equation*}
	\lim_{n\rightarrow+\infty}\int f|v_{n}|^{\gamma}=0.
	\end{equation*}
	In view of \eqref{alpha2} we get
	\[
	\int f|v|^{\gamma}=-\alpha_{\infty} f(\infty)\geq0.
	\]
	Since $v$ is an admissible function to the minimizing problem \eqref{lambdaex}, it follows from Corollary \eqref{cor1} that $\alpha_{\infty}=0$. Thus,
	\[
	\int |\nabla v|^{p}-\lambda^{*}\int h|v|^{p}=0 \quad \mbox{and} \quad \int f|v|^{\gamma}=0,
	\]
	and hence $v_n\to v$ in $E$. From 
	\begin{equation*}
	-\Delta_p v_{n} -\lambda_{n} h(x)|v_{n}|^{p-2}v_{n} = s_{\lambda^*}^{+}(v_{n})^{\gamma-p}f(x)|v_{n}|^{\gamma-2}v_{n}, \forall\ n
	\end{equation*}
	and Lemma \ref{lagranextre} we must conclude that $s_{\lambda^*}^{+}(v_{n})\to s\in (0,\infty)$, however, since 
	\[
	0>\hat{J}_{\lambda^{*}}^{+}=J_{\lambda^{*}}^{+}(v_{n})=-c_{p,\gamma}s_{\lambda^{*}}^{+}(v_{n})^{p}\left|\int|\nabla v_{n}|^{p}-\lambda^{*}\int h|v_{n}|^{p} \right|,
	\]
	we infer that $s_{\lambda^*}^{+}(v_{n})\to \infty$, a contradiction.
\end{proof}
The idea behind Lemma \ref{separataion} is to separate the minimizers of $\hat{J}_{\lambda^*}^+$ from $\mathcal{N}_{\lambda^*}^0$. Once we have such a separation we can prove
\begin{lemma}\label{lem2}
	For each $\mu\in(\lambda_{1},\lambda^{*})$, there exists $c_{\mu}<0$ such that
	\[
	\int f|v|^{\gamma}\leq c_{\mu}, \quad \forall v\in\overline{\Theta}_{\mu}^{+}\cap S.
	\]
\end{lemma}

\begin{proof}
	Suppose by contradiction that there exist $\mu\in (\lambda_{1},\lambda^{*})$ and a sequence $(v_{n})\subset \overline{\Theta}_{\mu}^{+}\cap S$ such that
	\begin{equation}\label{kj7}
	\lim_{n\rightarrow+\infty}\int f|v_{n}|^{\gamma}=0.
	\end{equation}
	Since $\|v_{n}\|=1$ we may assume up to a subsequence that $v_{n}\rightharpoonup v$ weakly in $E$. It follows from $(F_{\infty})$, \eqref{alpha2} and \eqref{kj7} that
	\begin{equation}\label{kj9}
	\int f|v|^{\gamma}=-\alpha_{\infty}f(\infty)\geq0.
	\end{equation}
	Arguing as in the proof of Lemma~\ref{l3} we conclude that $v\neq0$. Thus, one has
	\[
	\int |\nabla v|^{p}-\lambda^{*}\int h|v|^{p}\leq \liminf_{n\rightarrow+\infty}\left[\int |\nabla v_{n}|^{p}-\lambda\int h|v_{n}|^{p}+(\mu-\lambda^{*})\int h|v_{n}|^{p} \right]<0,
	\]
	which implies that
	\begin{equation}\label{kj8}
	\frac{\int |\nabla v|^{p}}{\int h|v|^{p}}<\lambda^{*}.
	\end{equation}
	Since $v\neq0$ and \eqref{kj9} holds, it follows that $v$ is an admissible function for the minimization problem \eqref{lambdaex}. Therefore, \eqref{kj8} contradicts the definition of $\lambda^{*}$ which finishes the proof.
\end{proof}

For given $\lambda\geq \lambda^{*}$ and $\mu\in(\lambda_{1},\lambda^{*})$ we introduce the following family of constrained minimization problems
\begin{equation}\label{hatj}
\hat{J}_{\lambda}^{+}(\mu):=\inf\left\{J_{\lambda}^{+}(v):v\in\Theta_{\mu}^{+}\cap S \right\}.
\end{equation}
In light of Lemma~\ref{lem2} one can conclude that $\hat{J}_{\lambda}^{+}(\mu)>+\infty$.

\begin{proposition}\label{p1}
	Let $\lambda\geq \lambda^{*}$ and $\mu\in(\lambda_{1},\lambda^{*})$. Then, there exists a minimizer $v_{\lambda}(\mu)$ of \eqref{hatj}.
\end{proposition}
\begin{proof}
	Let $(v_{n})\subset \Theta_{\mu}^{+}\cap S$ be a minimizing sequence of \eqref{hatj}, that is, $J_{\lambda}^{+}(v_{n})\rightarrow \hat{J}_{\lambda}^{+}(\mu)$, as $n\rightarrow+\infty$. Arguing as in the proof of Lemma~\ref{l3}, there exists $v\in E$, $v\neq0$, such that, up to a subsequence, $v_{n}\rightharpoonup v$ weakly in $E$. Since $(v_{n})\subset \Theta_{\mu}^{+}$ one has
	\begin{equation}\label{kj14}
	\int|\nabla v|^{p}-\mu\int h|v|^{p}\leq \liminf_{n\rightarrow+\infty}\left(\int|\nabla v_{n}|^{p}-\mu\int h|v_{n}|^{p}\right)\leq0.
	\end{equation}
	We claim that 
	\begin{equation}\label{kj15}
	\int f|v|^{\gamma}<0.
	\end{equation}
	In fact, let us suppose by contradiction that
	\begin{equation}\label{kj16}
	\int f|v|^{\gamma}\geq0.
	\end{equation}
	In view of \eqref{kj14} we have
	\[
	\frac{\int |\nabla v|^{p}}{\int h|v|^{p}}\leq \mu<\lambda^{*}.
	\]
	If \eqref{kj16} holds, then $v$ is an admissible function to the minimizing problem \eqref{lambdaex}, and we get a contradiction. Therefore, \eqref{kj15} holds. Hence, \eqref{kj14} and \eqref{kj15} imply that $v\in\overline{\Theta}_{\mu}^{+}$. It follows from Proposition \ref{liminf} that
	\[
	J_{\lambda}^{+}(v)\leq \liminf_{n\rightarrow+\infty}J_{\lambda}^{+}(v_{n})=\hat{J}_{\lambda}^{+}(\mu),
	\]
	which implies that $J_{\lambda}^{+}(v)=\hat{J}_{\lambda}^{+}(\mu)$, that is, $v:=v_{\lambda}(\mu)$ is a minimizer of \eqref{hatj}.
\end{proof}

Let us introduce the following sets:
\[
\mathcal{S}_{\lambda}(\mu):=\left\{v\in \overline{\Theta}_{\mu}^{+}\cap S:J_{\lambda}^{+}(v)=\hat{J}_{\lambda}^{+}(\mu) \right\},
\]
\[
\mathcal{S}_{\lambda}^{\partial}(\mu):=\left\{v\in\mathcal{S}_{\lambda}(\mu): \int|\nabla v|^{p}-\mu\int h|v|^{p}=0 \right\}.
\]

\begin{lemma}\label{aux1}
	Let $\lambda_{0}\geq \lambda^{*}$ and $\mu\in(\lambda_{1},\lambda^{*})$ be such that $\mathcal{S}_{\lambda_{0}}^{\partial}(\mu)=\emptyset$. Then, there exists $\varepsilon>0$ such that $\mathcal{S}_{\lambda}^{\partial}(\mu)=\emptyset$, for all $\lambda\in[\lambda_{0},\lambda_{0}+\varepsilon)$.
\end{lemma}

\begin{proof}
	Arguing by contradiction, let us suppose that for each $n\in\mathbb{N}$, there exist $\lambda_{n}\geq\lambda_{0}$ and $v_{n}:=v_{\lambda_{n}}^{+}(\mu)\in\mathcal{S}_{\lambda_{n}}^{\partial}(\mu)$. Moreover, suppose that $\lambda_{n}\rightarrow\lambda_{0}$ as $n\rightarrow+\infty$. Arguing as before, we may assume that, up to a subsequence, $v_{n}\rightharpoonup v$ weakly in $E$ and $v\neq0$. Arguing as in the proof of Proposition~\ref{p1} we conclude that $v\in\overline{\Theta}_{\mu}^{+}$. By using Poincar\'{e} inequality and Lemma~\ref{lem2}, we have that
	\[
	\left|(-J_{\lambda}^{+}(w))^{\frac{\gamma-p}{\gamma}}-(-J_{\lambda_{0}}^{+}(w))^{\frac{\gamma-p}{\gamma}} \right|=c_{p,\gamma}\frac{|\lambda-\lambda_{0}| \left|\int h|w|^{p} \right|}{\left|\int f|w|^{\gamma} \right|^{\frac{p}{\gamma}}}\leq c_{p,\gamma}\frac{|\lambda-\lambda_{0}|}{\lambda_{1}}\frac{1}{|c_{\mu}|^{\frac{p}{\gamma}}},
	\]
	for all $w\in\bar{\Theta}_{\mu}^{+}\cap S$.	In view of Proposition~\ref{liminf}, one has
	\begin{equation}
	J_{\lambda_{0}}^{+}(v)\leq \liminf_{n\rightarrow+\infty}J_{\lambda_{n}}^{+}(v_{n})=:\hat{J}^{+}<+\infty,
	\end{equation}
	for all $w\in\overline{\Theta}_{\mu}^{+}$ and $\lambda\geq \lambda_{1}$. Therefore, $J_{\lambda_{n}}^{+}\rightarrow J_{\lambda_{0}}^{+}(w)$ uniformly on $w\in\overline{\Theta}_{\mu}^{+}$, which implies that $\hat{J}^{+}=\hat{J}_{\lambda_{0}}^{+}(\mu)$. Thus, since $v\in\overline{\Theta}_{\mu}^{+}$, we conclude that $J_{\lambda_{0}}^{+}(v)=\hat{J}_{\lambda_{0}}^{+}(\mu)$. Hence, $v=v_{\lambda_{0}}(\mu)$ and
	\[
	\int|\nabla v|^{p}-\mu_{0}\int h|v|^{p}=0.
	\]
	Therefore, $v\in \mathcal{S}_{\lambda_{0}}^{\partial}(\mu)$ which is a contradiction and finishes the proof.
\end{proof}

Now, we are able to prove the existence of a positive solution to Problem \eqref{p} for $\lambda>\lambda^*$.

\begin{theorem}\label{first}
	There exists $\varepsilon>0$ such that for any $\lambda\in(\lambda^{*},\lambda^{*}+\varepsilon)$, Problem \eqref{p} admits a positive weak solution. 
\end{theorem}
\begin{proof}
	In view of Lemma~\ref{separataion}, there exists $\mu_{0}\in(\lambda_{1},\lambda^{*})$ such that any minimizer $v_{\lambda^{*}}\in\Theta_{\lambda^{*}}^{+}$ of $\hat{J}_{\lambda^{*}}^{+}$ satisfies $H_{\mu_{0}}(v_{\lambda^{*}})<0$. Thus, we have $\mathcal{S}_{\lambda}^{\partial}(\mu_{0})=\emptyset$. Hence, it follows from Lemma~\ref{aux1} that there exists $\varepsilon>0$ such that $\mathcal{S}_{\lambda}^{\partial}(\mu_{0})=\emptyset$, for all $\lambda\in[\lambda^{*},\lambda^{*}+\varepsilon)$. In light of Proposition~\ref{p1}, for any $\lambda\in(\lambda^{*},\lambda^{*}+\varepsilon)$ there exists a minimizer of \eqref{hatj}, i.e., there exists $v_{\lambda}(\mu_{0})\in \Theta_{\mu_{0}}^{+}$ such that $J_{\lambda}^{+}(v_{\lambda}(\mu_{0}))=\hat{J}_{\lambda}^{+}(\mu_{0})$. Therefore, Lemma \ref{yavdat} implies that $u_{\lambda}:=s_{\lambda}^{+}(v_{\lambda}(\mu_{0}))v_{\lambda}(\mu_{0})$ is a weak solution of Problem \eqref{p} for $\lambda\in(\lambda^{*},\lambda^{*}+\varepsilon)$. Since $|u_{\lambda}|\in\Theta_{\mu_{0}}^{+}$ and $J_{\lambda}^{+}(u_{\lambda})=J_{\lambda}^{+}(|u_{\lambda}|)$, we may assume that $u_{\lambda}\geq0$ in $\mathbb{R}^{N}$. By using Strong Maximum Principle we conclude that $u_{\lambda}>0$ in $\mathbb{R}^{N}$. This finishes the proof. 
\end{proof}

\begin{remark}
	It is worthwhile to mention that the solution obtained in Proposition~\ref{first} may depend on the parameter $\mu\in(\lambda_{1},\lambda^{*})$. A natural question arises: What is the dependence of the parameter? By similar arguments to \cite[Corollary 3.4]{yavkay} one can deduce that at least locally the set of minimizers $\mathcal{S}_{\lambda}(\mu)$ does not depend on the parameter $\mu\in(\lambda_{1},\lambda^{*})$.
\end{remark}



\section{Second Solution for $\lambda>\lambda^*$}\label{s5}

In this Section we complete the proof of Theorem \ref{thm}. To this end we look for a second solution for Problem \eqref{p} when $\lambda>\lambda^*$. For this purpose, we adapt the ideas introduced in \cite[Section 4]{yavkay}. In fact, the Mountain Pass geometry is obtained by similar calculations and we omit the proof. The problem here is the lack of compactness inherit from the unbounded domain. For this reason, it is necessary to use new techniques in order to show that (P.-S.) sequences converge strongly to weak solutions. In view of Lemma~\ref{separataion}, there exists $\mu_{0}\in(\lambda_{1},\lambda^{*})$ such that any minimizer $v_{\lambda^{*}}$ of $\hat{J}_{\lambda^{*}}^{+}$ satisfies $H_{\mu_{0}}(v_{\lambda^{*}})<0$. 

Let $\varepsilon>0$ be the parameter obtained in Proposition~\ref{first} and $\lambda\in(\lambda^{*},\lambda^{*}+\varepsilon)$. We define
 \[
  \mu^{\lambda}:=\sup\left\{\mu\in(\mu_{0},\lambda^{*}):\hat{J}_{\lambda}^{+}(\mu)=\hat{J}_{\lambda}^{+}(\mu_{0}) \right\}.
 \]
Notice that if $\lambda\in(\lambda^{*},\lambda^{*}+\varepsilon)$, $\mu\in(\lambda_{1},\lambda^{*})$ and $v\in\mathcal{S}_{\lambda}^{\partial}(\mu^{\lambda})$, then $|v|\in\mathcal{S}_{\lambda}^{\partial}(\mu^{\lambda})$. For $\lambda\in(\lambda^{*},\lambda^{*}+\varepsilon)$, let $v_{\lambda}\in\mathcal{S}_{\lambda}^{\partial}(\mu^{\lambda})$ be a fixed nonnegative function and let $u_{\lambda}\in\Theta_{\mu_{0}}^{+}$ be the positive solution which has been obtained in Proposition~\ref{first}. Let us define
 \[
  c_{\lambda}:=\inf_{\eta\in\Gamma_{\lambda}}\max_{t\in[0,1]}\Phi_{\lambda}(\eta(t)),
 \]
where
 \[
  \Gamma_{\lambda}:=\left\{\eta\in C([0,1], E):\eta(0)=u_{\lambda}, \ \eta(1)=v_{\lambda} \right\}.
 \]
By the same ideas used in \cite{yavkay} we can obtain some auxiliary lemmas which imply the mountain pass geometry. We summarize the results in the following Proposition:

\begin{proposition}\label{prop}
	For any $\lambda\in(\lambda^{*},\lambda^{*}+\varepsilon)$, the following facts hold: 
	 \begin{itemize}
	 	\item[(i)] $\mu_{0}<\mu^{\lambda}<\lambda^{*}$;
	 	\item[(ii)] $\hat{J}_{\lambda}^{+}(\mu^{\lambda})=\hat{J}_{\lambda}^{+}(\mu_{0})$ and $\mathcal{S}_{\lambda}^{\partial}(\mu^{\lambda})\neq\emptyset$;
	 	\item[(iii)] There exists $j_{\lambda}$ such that $\Phi_{\lambda}\geq j_{\lambda}>\hat{J}_{\lambda}^{+}(\mu_{0})$, for all $u\in\partial\Theta_{\mu_{0}}^{+}$;
	 	\item[(iv)] For any $\eta\in\Gamma_{\lambda}$, there exists $t_{0}\in(0,1)$ such that $\eta(t_{0})\in\partial\Theta_{\mu_{0}}^{+}$;
	 	\item[(v)] There exists $\bar{\eta}\in\Gamma_{\lambda}$ such that $H_{\lambda^{*}}(\bar{\eta}(t))<c<0$, for all $t\in[0,1]$;
	 	\item[(vi)] $\hat{J}_{\lambda}^{+}(\mu_{0})<c_{\lambda}<0$.
	 \end{itemize}
\end{proposition}
\begin{remark} Note that condition $(vi)$ of Proposition \ref{prop} gives the desired mountain pass geometry to $\Phi_\lambda$ with respect to $c_\lambda$.
\end{remark}
We emphasize that the main problem here is to overcome the difficulty imposed by the lack of compactness. Precisely, it is not clear that the energy functional $\Phi_{\lambda}$ satisfies the Palais-Smale condition at level $c\in\mathbb{R}$, i.e., if (P.-S.) sequences admit a strong convergent subsequence. Now, we prove that if this fact holds then we have the existence of a positive solution with energy at a mountain pass level. 

\begin{theorem}\label{second}
	Let $\lambda\in(\lambda^{*},\lambda^{*}+\varepsilon)$ and suppose that $\Phi_\lambda$ satisfies the (P.-S.) condition at the level $c_\lambda$. Then, Problem \eqref{p} admits a positive weak solution $u_{\lambda}$ such that $\Phi_{\lambda}(u_{\lambda})=c_{\lambda}$. 
\end{theorem}

\begin{proof}
	Let $(\eta_{n})\subset C([0,1],E)$ be a sequence of paths such that
	 \[
	  \lim_{n\rightarrow+\infty}\max_{t\in[0,1]}\Phi_{\lambda}(\eta_{n}(t))=c_{\lambda}.
	 \]
	We may assume without loss of generality that $\eta_{n}$ is nonnegative in $\mathbb{R}^{N}$ for all $n\in\mathbb{N}$. For any $\epsilon>0$ consider the set
	 \[
	  \eta_{n,\epsilon}=\left\{u\in E:\inf_{t\in[0,1]}\|u-\eta_{n}(t)\|\leq\epsilon \right\}\cap K_{c_{\lambda},2\epsilon},
	 \]
	where
	 \[
	  K_{c_{\lambda},2\epsilon}=\left\{u\in E:|\Phi_{\lambda}(u)-c_{\lambda}|\leq 2\epsilon \right\}.
	 \]
	In view of \cite[Theorem E.5]{kuzin}, there exists a sequence $(u_{n})\subset E$ which satisfies
	 \begin{equation}\label{kj18}
	  \Phi_{\lambda}(u_{n})\rightarrow c_{\lambda}, \quad \Phi_{\lambda}^{\prime}(u_{n})\rightarrow0 \quad \mbox{and} \quad \inf_{t\in[0,1]}\|u_{n}-\eta_{n}(t)\|\rightarrow0, \quad \mbox{as } n\rightarrow+\infty.
	 \end{equation}
	By hypothesis, up to a subsequence, $u_{n}\rightarrow u_{\lambda}$ strongly in $E\backslash\{0\}$, $\Phi_{\lambda}(u_{\lambda})=c_{\lambda}$ and $\Phi_{\lambda}^{\prime}(u_{\lambda})=0$. Moreover, $u_{\lambda}\geq0$ in $\mathbb{R}^{N}$. Therefore, Strong Maximum Principle implies that $u_{\lambda}>0$ in $\mathbb{R}^{N}$, which finishes the proof.
\end{proof}

In view of the preceding Proposition, it remains to prove that $\Phi_{\lambda}$ satisfies the Palais-Smale condition. For this purpose, the hypothesis $(F_{\infty})$ plays a very important role in our technique. 

\begin{proposition}\label{la1c}	
	Suppose that $(u_n)\subset E\setminus \{0\}$ is a (P.-S.) sequence at level $c<0$, 
	i.e.
	\begin{equation}\label{kj111}
	\Phi_\lambda(u_n)\to c< 0,~~ \Phi'(u_n)\to 0,\ n\to \infty.
	\end{equation}
	Assume that $B$ is an open ball contained in $\Omega_f^0$. If $\lambda$ is not an eingenvalue of $-\Delta_p$ over $B$, then $(u_n)$ has a strong convergent subsequence with limit point $ 
	u_{\lambda}\in E\setminus \{0\}$ satisfying $\Phi_\lambda(u_{\lambda})=c$ and 
	$\Phi'_\lambda(u_{\lambda})=0$. 	
\end{proposition}

\begin{proof} We claim that the sequence $(\|u_n\|)$ is bounded. Indeed, suppose on the contrary that, up to a subsequence, we have $\|u_n\|\to \infty$, as $n\to \infty$. Write $v_n=u_n/\|u_n\|$ and suppose without loss of generality that $v_n \rightharpoonup v$ weakly in $E$, $\int h|v_n|^p\to \int h| v|^p$ and $v_n\to v$ strongly in $L^\gamma_{loc}(\mathbb{R}^N)$. It follows from \eqref{kj111} that
	\begin{equation}\label{M1}
	\frac{H_\lambda(u_n)}{p}-\frac{F(u_n)}{\gamma}=c+o(1),
	\end{equation}
	and
	\begin{equation}\label{M2}
	\frac{H_\lambda(u_n)-F(u_n)}{\|u_n\|^p}=o(1).
	\end{equation}
	We first prove that $v\neq 0$. In fact, combine \eqref{M1} with \eqref{M2} to obtain
	\begin{equation}\label{M3}
	\frac{\gamma-p}{p\gamma}	\|u_n\|^{\gamma-p}F(v_n)=\frac{c+o(1)}{\|u_n\|^p}.
	\end{equation}
	Since $c<0$ we conclude that $F(v_n)<0$ for $n$ sufficiently large. From \eqref{M2} it follows that $\int |\nabla v_n|^p<\lambda\int h|v_n|^p$ for $n$ sufficiently large. If $v=0$ then $\int |\nabla v_n|^p\to 0$ as $n\to \infty$ and hence $\int |v_n|^\gamma\to 1$ as $n\to\infty$. From \eqref{alpha1} it follows that $\alpha_\infty=1$ and from \eqref{alpha2} we conclude that 
	\begin{equation*}
	\limsup_{n\to \infty}F(v_n)=f(\infty)<0,
	\end{equation*}
	which contradicts \eqref{M3}. Therefore, $v\neq0$. Now observe that
	\begin{equation}\label{M4}
	-\Delta_pv_n-\lambda h|v_n|^{p-2}v_n-\|u_n\|^{\gamma-p}f|v_n|^{\gamma-2}v_n=o(1).
	\end{equation}
	Since $(v_n)$ is bounded, we obtain from \eqref{M4} that $f|v_n|^{\gamma-2}v_n\to 0$, as $n\to \infty$. Thus, the support of $v$ is contained on $\Omega\setminus (\Omega^+\cup \Omega^-)$. Once $\|v_n-v\|$ is bounded, by choosing $v_n-v$ as test function in \eqref{M4}, we conclude that
	\begin{equation}\label{M5}
	\lim_{n\to \infty}\left[-\Delta_p v_n(v_n-v)-\lambda h|v_n|^{p-2}v_n(v_n-v)-\|u_n\|^{\gamma-p}f|v_n|^{\gamma-2}v_n(v_n-v)\right]= 0.
	\end{equation}
	Notice that $f|v_n|^{\gamma-2}v_nv=0$, for $n\in\mathbb{N}$. Thus, it follows from \eqref{M3} that $\|u_n\|^{\gamma-p}F(v_n)=o(1)$. In view of \eqref{M5} we have	
	\begin{equation}\label{M6}
	\lim_{n\to \infty}-\Delta_p v_n(v_n-v)=\lim_{n\to \infty}\left[\lambda h|v_n|^{p-2}v_n(v_n-v)-\|u_n\|^{\gamma-p}F(v_n)\right]=0.
	\end{equation}
	By using the estimates
	\begin{equation}\label{inequality}
	-\Delta_p (v_n-v)(v_n-v)\ge\left\{
	\begin{aligned}
	c_p\int |\nabla v_n-\nabla v|^p,  &&p\ge 2, \\
	c_p\int\frac{ |\nabla v_n-\nabla v|^2}{(|\nabla v_n|+|\nabla v|)^{p-2}}, &&1<p<2,
	\end{aligned}
	\right.
	\end{equation}	
	jointly with \eqref{M6}, we conclude that $v_n\to v$ strongly in $D^{1,p}(\mathbb{R}^{N})$. Thus, one has
	\begin{equation}\label{M7}
	-\Delta_p v-\lambda h|v|^{p-2}v=\lim _{n\to \infty}\|u_n\|^{\gamma-p} f|v_n|^{\gamma-2}v_n.
	\end{equation}	
	By taking $w\in E$ with compact support contained in $B$ as test function in \eqref{M7}, we conclude that $\lambda$ is an eigenvalue to $-\Delta_p$ over $B$, which is not possible. Therefore, $(\|u_n\|)$ is bounded. We may assume, without loss of generality, that $u_n\rightharpoonup u_{\lambda}$ weakly in $E$, $\int h|u_n|^p\to \int h|u_{\lambda}|^p$ and $u_n\to u_{\lambda}$ strongly in $L^p_{loc}(\mathbb{R}^N)$ and $L^\gamma_{loc}(\mathbb{R}^N)$. If $u_{\lambda}=0$, then from \eqref{M3} we get a contradiction and hence $u_{\lambda}\neq 0$. Hence, for each $\varphi\in C_0^\infty(\mathbb{R}^N)$ there holds
	\begin{equation*}
	\lim_{n\to \infty}\left[-\Delta_p u_n(\varphi(u_n-u_{\lambda}))-\lambda h|u_n|^{p-2}u_n(\varphi(u_n-u_{\lambda}))-f|u_n|^{\gamma-2}u_n(\varphi(u_n-u_{\lambda}))\right]= 0,
	\end{equation*}
	which implies that
	\begin{equation}\label{M8}
	\lim_{n\to \infty}-\Delta_p u_n\left[\varphi(u_n-u_{\lambda})\right]=0, \quad \forall\ \varphi\in C_0^\infty(\mathbb{R}^N).
	\end{equation}
	Observe that
	\begin{equation*}
	-\Delta_p u_n\left[\varphi(u_n-u_{\lambda})\right]=\int \varphi|\nabla u_n|^{p-2}\nabla u_n(\nabla u_n-\nabla u_{\lambda})+\int |\nabla u_n|^{p-2}\nabla u_n\nabla \varphi(u_n-u_{\lambda}).
	\end{equation*}
	Thus, one has
	\begin{equation}\label{M9}
	\lim_{n\to \infty}-\Delta_p u_n(\varphi(u_n-u_{\lambda}))=\lim_{n\to \infty}\int \varphi|\nabla u_n|^{p-2}\nabla u_n(\nabla u_n-\nabla u_{\lambda}),\quad \forall\ \varphi\in C_0^\infty(\mathbb{R}^N).
	\end{equation}
	We combine \eqref{inequality}, \eqref{M8} and \eqref{M9} to obtain that $|\nabla u_n|\to |\nabla u_{\lambda}|$ in $L^p_{loc}(\mathbb{R}^N)$ and hence
	\begin{equation*}
	\Phi_\lambda'(u_{\lambda})\varphi=\lim\Phi_\lambda'(u_n)\varphi=0,\quad \forall\ \varphi\in C_0^\infty(\mathbb{R}^N).
	\end{equation*}
	Since $C_0^\infty(\mathbb{R}^N)$ is dense in $E$, we conclude that $\Phi'_\lambda(u_{\lambda})=0$. Now we claim that $u_n\to u_{\lambda}$ in $E$. Indeed, from $\lim_{n\to \infty}\Phi'_\lambda(u_n)u_n=0$ we conclude from \eqref{alpha} and \eqref{alpha2} that 
	\begin{equation*}
	H_\lambda(u_{\lambda})+\beta_\infty=F(u_{\lambda})+\alpha_\infty f(\infty).
	\end{equation*}
	Once $\Phi'_\lambda(u_{\lambda})=0$ it follows that $\beta_\infty=\alpha_\infty f(\infty)$. Therefore, $\beta_\infty=\alpha_\infty=0$ which implies the strong convergence $u_n\to u_{\lambda}$ in $E$ and consequently $\Phi_\lambda(u_{\lambda})=c$.
\end{proof}
Now we prove the main result of this work
\begin{proof}[Proof of Theorem \ref{thm}] The inequality $\lambda^*>\lambda_1$ follows from Lemma \ref{l1}. The existence of $u_\lambda$ is a consequence of Theorems \ref{exisbeforlam}, \ref{l3} and \ref{first}. The second solution $w_\lambda$ follows from Theorem \ref{exisbeforlam} when $\lambda\in(\lambda_1,\lambda^*)$ and from Proposition \ref{la1c} combined with Theorem \ref{second} in the case where $\lambda\in(\lambda^*,\lambda^*+\varepsilon)$.
	
\end{proof}









\end{document}